\documentclass{article}
\usepackage[]{amsmath}
\usepackage{amsfonts}
\usepackage[]{geometry}
\usepackage[]{amsthm}
\usepackage{algpseudocode}
\usepackage{algorithm}
\usepackage[texencoding=ascii, style=alphabetic, maxbibnames=99]{biblatex}
\usepackage[]{graphicx}
\usepackage[]{caption}
\usepackage[colorlinks]{hyperref}

\newtheorem{theorem}{Theorem}

\newtheorem{lemma}{Lemma}
\newtheorem{proposition}{Proposition}
\newtheorem{conjecture}{Conjecture}

\theoremstyle{definition}
\newtheorem{definition}{Definition}

\newcommand{\oeis}[1]{\href{https://oeis.org/#1}{#1}}

\title{The Comma Sequence is Finite in Other Bases}
\author{Robert Dougherty-Bliss \and Natalya Ter-Saakov}
\date{\today}

\begin{document}

\maketitle

\begin{abstract}
    \noindent The comma sequence (1, 12, 35, 94, \dots) is the lexicographically
    earliest sequence such that the difference of consecutive terms equals the
    concatenation of the digits on either side of the comma separating them. The
    behavior of a ``generalized comma sequence'' depends on the base the numbers
    are written in, as well as the sequence's initial values. We provide a
    computational proof that all comma sequences in bases 3 through 633 are finite. Relying on a combinatorial conjecture, Angelini et
    al.~estimated that the final element of a comma sequence in base $b$ should
    be roughly $\exp(O(b))$. We prove their conjecture, but provide evidence
    that the correct estimate is actually $\exp(O(b \log b))$.
\end{abstract}

\section{Introduction}
\label{sec:Introduction}

In 2006, \'Eric Angelini submitted \oeis{A121805}, the ``comma sequence,'' to
the OEIS \cite{oeis}. It is defined by a peculiar rule, best illustrated by
listing the initial terms and their differences:
\begin{align*}
    1\quad &, \quad 12\quad , \quad 35\quad , \quad 94\quad , \quad 135\quad , \quad 186\quad , \quad 248\quad , \quad 331\quad , \quad 344, \quad \dots \\
           &11\quad , \quad 23\quad , \quad 59\quad , \quad 41\quad , \quad\ \ 51\quad , \quad\ 62\quad , \quad\ 83\quad , \quad\ 13, \quad \dots
\end{align*}
The comma sequence is the lexicographically earliest sequence of positive
integers such that the difference of consecutive terms equals the concatenation
of the digits on either side of the comma separating them. For example, the
sequence contains $12, 35$, which has a difference of $35 - 12 = 23$, which
happens to be the concatenation of $2$ (the digit to the left of the comma) and
$3$ (the digit to the right of the comma). The amazing fact about the comma
sequence is that it contains exactly 2,137,453 terms. It reaches $99999945$ and
then terminates.

A (generalized) comma sequence in base $b$ is the lexicographically earliest
sequence which satisfies the comma rule with a particular fixed initial value.
In \cite{neil}, Angelini et al.~studied generalized comma sequences and obtained
a number of results we will generalize. First, they proved that all comma
sequences in base 3 are finite. We have extended this result to more bases.

\begin{theorem}
    \label{theorem-finite}
    Comma sequences with arbitrary positive initial values are finite in bases 3 through 633.
\end{theorem}

The proof of this theorem is a parallelized computation designed in the Go
programming language and carried out on a computing cluster. We will give the
mathematical details of this proof in Section~\ref{details}, and the
computational details in Section~\ref{compdetails}.

The obvious conjecture is that all comma sequences in all bases are finite. This
remains open, but Angelini and company introduced a simple random model which
attempts to explain the expected duration of comma sequences in different bases.
They made a combinatorial conjecture about this model, which we have confirmed.

\begin{theorem}
    \label{theorem-difference}
    The number $D(b)$ of comma sequences in base $b \geq 2$ with initial values
    in $[b^m - b^2, b^m)$ (with $m \geq 2$) that do not reach $b^m$ is
    independent of $m$ and has ordinary generating function
    \begin{equation}
        \label{tri-gf}
        \sum_{b = 2}^\infty D(b) t^b =
        (1 - t)^{-1} \left(\sum_{b = 1}^\infty \frac{t^{b(b + 3) / 2}}{(1 - t^b)} - t^2\right) = t^3 + 2t^4 + 4t^5 + 5t^6 + \cdots
    \end{equation}
\end{theorem}

On the basis of this theorem, the authors of \cite{neil} (with help from
V\'aclav Kot\u{e}\u{s}ovec) argued that the final element of a base $b$ comma
sequence should be roughly $e^{2b}$. More precisely, they estimated that a comma
sequence ought to survive about $2b / \log(2b)$ ``danger intervals,'' which we
will define below. In Section~\ref{graph-model}, we will define a different
probabilistic model which predicts that comma sequences should survive about $b
/ 2 + 1$ intervals, which translates to an estimate length of $b^{b/2 + 1} = e^{O(b \log b)}$. This estimate is more accurate for large bases, but depends
on a probabilistic conjecture that we cannot prove. We hope that some
expert can translate at least one of these models into a rigorous
argument that all comma sequences are finite.

\section{Background and examples}

Because comma sequences and their generalizations are not widely known, in this
section we summarize some relevant background \cite{neil} and give a number of
examples. We begin with a more explicit definition of generalized comma
sequences.

\begin{definition}
    The generalized comma sequence in base $b$ with initial condition $v$ is the
    sequence $a(n)$ defined as follows: $a(1) = v$; for $n > 1$, if $x$ is the
    least significant base $b$ digit of $a(n - 1)$, then
    \begin{equation}
        \label{pick-y}
        a(n) = a(n - 1) + bx + y,
    \end{equation}
    where $y$ is the most significant digit of $a(n)$ and is the smallest such
    $y$, if any exists. If no such $y$ exists, then the sequence terminates.
\end{definition}

\emph{The} comma sequence corresponds to $v = 1$ and $b = 10$. The first ten
terms of the base $10$ comma sequences with initial values $v = 1, 2, \dots, 10$
are as follows:
\begin{align*}
    &1, 12, 35, 94, 135, 186, 248, 331, 344, 387 \dots \\
    &2, 24, 71, 89, 180, 181, 192, 214, 256, 319 \dots \\
    &3, 36 \text{ (terminates)} \\
    &4, 48, 129, 221, 233, 265, 318, 402, 426, 490 \dots \\
    &5, 61, 78, 159, 251, 263, 295, 348, 432, 456 \dots \\
    &6, 73, 104, 145, 196, 258, 341, 354, 397, 471 \dots \\
    &7, 85, 136, 197, 269, 362, 385, 439, 534, 579 \dots \\
    &8, 97, 168, 250, 252, 274, 317, 390, 393, 427 \dots \\
    &9, 100, 101, 112, 133, 164, 206, 268, 351, 364 \dots \\
    &10, 11, 23, 58, 139, 231, 243, 275, 328, 412 \dots
\end{align*}
These sequences are all finite. Their lengths are recorded in \oeis{A330128}.

The first ten terms of the base $3$ comma sequences with initial values $v = 1,
2, 3$ are as follows (given in base $3$):
\begin{align*}
    &(1)_3, (12)_3, (110)_3, (111)_3, (122)_3, (221)_3, (1002)_3, (1100)_3, (1101)_3, (1112)_3 \dots \\
    &(2)_3, (100)_3, (101)_3, (112)_3, (211)_3 \text{ (terminates)} \\
    &(10)_3, (11)_3 \text{ (terminates)}
\end{align*}
These sequences are also all finite; the first lasts for seventeen terms.

The numbers which do not have successors, meaning integers for which comma
sequences terminate, are called ``landmines.'' They are precisely the integers
in base $b$ of the form
\begin{equation*}
    (b - 1, b - 1, \dots, b - 1, x, y)_b
\end{equation*}
where $x$ and $y$ are nonzero digits which sum to $b - 1$ \cite[Theorem~5.2]{neil}.
(We use the notation $(d_m, d_{m - 1}, \dots, d_0)_b$ to represent the integer $d_m
b^m + d_{m-1} b^{m - 1} + \cdots + d_0$ in base $b$ where $d_i$ are digits base
$b$.) For example, in base $10$, the integers 45, 99972, and 9999918 are all
landmines.

Landmines live in the following small intervals just to the left of powers of
$b$.

\begin{definition}
    The base $b$ \emph{danger intervals} are the intervals $[b^m - b^2, b^m)$
    for $m \geq 2$.
\end{definition}

Given a positive integer $x$, call the set of integers which \emph{could} follow
$x$ in a base $b$ comma sequence $F_b(x)$. A landmine is characterized by
$|F_b(x)| = 0$. In general, $|F_b(x)| \leq 2$. To see this, note that the next
term of a comma sequence following $x$ is $x + d$ for some two digit number $d$,
where the most-significant digit of $d$ is determined by $x$. That leaves only
two possibilities: a carry occurs or does not. In our ``strict'' definition, we
resolve the possible ambiguity of which $d$ to pick by requiring that the
smallest possible successor is chosen.

In the other direction, it can be shown that that every integer greater than $b^2$ 
belongs to precisely one $F_b(x)$.  
Of them, the only ones which are \emph{never}
chosen as a successor in a base $b$ comma sequence are of the form
\begin{equation*}
    (d, 0, 0, \dots, 0)_b,
\end{equation*}
where $d > 1$ is a base $b$ digit \cite[Theorems~5.4, 5.5]{neil}.

From a more graphical perspective, imagine a directed graph $C_b$ on the
positive integers which contains the edge $i \to j$ provided that $j$ would
satisfy \eqref{pick-y} (though would perhaps be the larger of two options) in
base $b$. What we have just said amounts to the observation that $C_b$ has
out-degree at most $2$ and in-degree at most $1$, and further that $C_b$ is a
union of finitely many disjoint (directed) trees which partition the integers.
As a corollary, it contains an infinite path, meaning that if comma sequences
were allowed to pick either of two possible successors, then there would be at
least one infinite comma sequence in every base \cite[Theorem~5.6]{neil}.

\section{The computational proof}
\label{details}

Neil Sloane and Giovanni Resta communicated to us the rough idea behind their
proof that all comma sequences in base $3$ are finite. We will describe their
idea and its generalization to arbitrary bases, illustrate some optimizations
that we made, and give conjectured runtimes for applying the algorithm in
arbitrary bases. This section contains primarily mathematical details, while Section~\ref{compdetails} contains more details about the computation and its implementation.

\begin{definition}
    Given a base $b$ digit $d$, an integer $0 \leq u < b^2$, and a positive
    integer $k$, we define $(d, u, k)$ to represent the integer $d \cdot b^k -
    u$. 
\end{definition}

Comma sequences can be thought of as the result of repeatedly applying a certain
complicated map to the points $(d, u, k)$. For example, in the \emph{original}
comma sequence, we see the values $a(4) = 94$, $a(5) = 135$, later $a(6) = 186$,
and then $a(7) = 248$. We would describe this as
\begin{equation*}
    (1, 6, 2) \to (2, 14, 2) \to (3, 52, 2).
\end{equation*}
Note that we omitted $a(5) = 135$ in this diagram. This is because we always
take the smallest possible value of $u$ for a given $d$ and $k$, and any comma sequence in base 10 which
reaches $135$ (alias $(2, 65, 2)$) will then reach $186$ (alias $(2, 14, 2)$).
More generally, for any base $b$ and digit $d$, there is a finite set of
``minimal $u$'s'' which all other $u$'s lead to, and these values are
independent of $k$. The precise statement of this fact and its (annoyingly
technical) proof are the following lemma.

\begin{lemma}[Minimal $u$'s]
    For any base $b > 2$, any integer $k > 2$ and any nonzero base $b$
    digit $d$, there exists a smallest finite set $U(b, d)$ that satisfies the following property: If
    any base $b$ comma sequence has a value in $[d \cdot b^k -
    b^2, d \cdot b^k)$, then its final value in this interval is $d \cdot b^k - u$ for
    some $u \in U(b, d)$.

    If $d \neq 1$, then
    \begin{equation*}
        U(b, d) = \{(r, s)_b \mid r + s < b,\ 0 < s < b\} \cup
        \{(r, s)_b \mid r + s = b,\ s < d\}.
    \end{equation*}

    If $d = 1$, then
    \begin{equation*}
        U(b, 1) = \{(r, s)_b \mid r + s \leq b,\ 0 < s < b\}.
    \end{equation*}
\end{lemma}

\begin{proof}
    Begin with $d > 1$. The final value of a base $b$ comma sequence in $[d
    \cdot b^k - b^2, d \cdot b^k)$ will have the form $(d - 1, b - 1, \dots, b -
    1, x, y)_b$ for some $0 \leq x, y < b$. Because this is the last term in the
    interval, the next term is
    \begin{equation*}
        (d - 1, b - 1, \dots, b - 1, x, y) + (y, d)_b,
    \end{equation*}
    The relationship between $(x, y)$ and $(r, s)$ is
    \begin{equation*}
        (x, y)_b + (r, s)_b = b^2.
    \end{equation*}
    There are two cases: First, $x + y \geq b$. Since $y > 0$, we have $(r, s) =
    (b - 1 - x, b - y)$ and the inequality $x + y \geq b$ is equivalent to $r +
    s \leq b - 1$ and $s > 0$.

    Second, $x + y = b - 1$ and $y + d \geq b$. Note that
    \begin{equation*}
        (d - 1, b - 1, \dots, b - 1, x, y) + (y, d - 1)_b
    \end{equation*}
    is also no longer in the interval, so $y + d - 1 \geq b$. This again implies $y >
    0$, and so $(r, s) = (b - 1 - x, b - y)$, and the conditions
    \begin{align*}
        x + y &= b - 1 \\
        y + d &\geq b + 1
    \end{align*}
    are equivalent to
    \begin{align*}
        r + s &= b \\
        s &< d.
    \end{align*}

    The reasoning for $d = 1$ is similar. The only difference is in the second
    case, when the inequality $y + d - 1 \geq b$ becomes $y + b - 1 \geq b$. The
    rest of the argument is identical.
\end{proof}

\begin{definition}
    The digraph $G_b$ consists of vertices labeled $(d, u, k)$ with base $b$
    digit $d$, a minimal $u$ with respect to $d$, and an integer $k \geq 0$. The
    edge $(d, u, k) \to (d', u', k')$ exists if the latter is the immediate
    image of the former under the comma map.
\end{definition}

Our main goal is to prove that $G_b$ does not have infinite paths. In the
remainder of this section we will show how this reduces to a finite problem. To
begin, we must describe a computational shortcut.

\paragraph{Computing the transform}
Given $(d, u, k)$ which is not a landmine,
we want to determine the relation
\begin{equation*}
    (d, u, k) \to (d', u', k').
\end{equation*}
For example, if we have a base $10$ sequence at the value $(6, 8, 3)$ (also
known as $5992$), where does it go next?
Both $d'$ and $k'$ are easy: $d' = d + 1$ if $d < b - 1$, and $d' = 1$
otherwise. Similarly, $k' = k$ if $d < b - 1$, and $k' = k + 1$ otherwise. So, as a partial answer, $(6, 8, 3) \to (7, u', 3)$.
The main difficulty is determining $u'$.

Starting from $5992$, the next few differences of the sequence would be $26,
86,46,06$, and $66$. After that, they cycle through the same values until the
comma sequence reaches a number which begins with a $6$. In general, if our
current value is $(d, u, k)$, then Angelini and friends proved that the
difference between consecutive terms cycle through the arithmetic progression
\begin{equation}
    \label{cycle}
    S_b(d, u) = \{((md - u ) \bmod b) b + d \mid 0 \leq m < b / \gcd(b, d)\}
\end{equation}
until they reach $(d', u', k')$ (see \cite[sec.~6]{neil} and \oeis{A121805}).
Therefore, to compute $u'$, we start at $d \cdot b^k - u$ and add the elements
of $S_b$ in order, repeating until we get as close to $(d + 1) b^k$ as possible
without going over.

The distance between $(d + 1) b^k$ and $d b^k - u$ is $b^k +
u$. Using entire copies of $S_b(d, u)$, we can get as close as
\begin{equation}
    \label{S-mod}
    (b^k + u) \bmod \sum S_b(d, u)
\end{equation}
where $\sum S_b(d, u)$ is the sum of the entries of the cycle. After this, we
must manually add elements of $S_b(d, u)$, starting from the beginning, to close
the gap.

\paragraph{Making the problem finite} The key point is \eqref{S-mod}. It is
well-known that $b^k \bmod m$, for integers $b$ and $m \geq 1$, is eventually
periodic in $k$. (Indeed, any C-finite sequence taken with a fixed modulus is
eventually periodic \cite{pisano, carmichael, concrete}.) This means
that, for each cycle $S_b(d, u)$, there exists a period $l(d, u)$ and a
nonnegative integer $k_0(d, u)$ such that
\begin{equation*}
    b^k \equiv b^{k + l(d, u)} \bmod{\sum S_b(d, u)}
\end{equation*}
for all $k \geq k_0(d, u)$. In particular, for sufficiently large $k$ we only
need to compute \eqref{S-mod} for values of $k \bmod l(d, u)$. If we set
\begin{equation}
    \label{lcm}
    L(b) = \mathrm{lcm} \{l(d, u) \mid 1 \leq d < b, \quad u \in U(b, d) \cup \{0\}\},
\end{equation}
then $b^k \bmod \sum S_b(d, u)$ is eventually periodic with period $L(b)$ for
\emph{all} $(d, u)$. This implies that, to determine the behavior of comma
sequences with sufficiently large $k$, we only need to consider points $(d, u,
k)$ where $k$ is the representative of an \emph{equivalence class} modulo
$L(b)$.

The technical point of forcing $k$ to be ``sufficiently large'' is not important
for our proof. If there were a truly infinite comma sequence, then its values
would pass by arbitrarily large $b^k$'s, meaning that the periodicity of
\eqref{S-mod} would kick in at some point.

\begin{definition}
    The digraph $G_b'$ consists of vertices $(d, u, \kappa)$ where $d$ is a
    base-$b$ digit, $u$ is an element of $U(b, d)$, and $[\kappa]$ is an
    equivalence class mod $L(b)$. The edge $(d, u, [\kappa]) \to (d', u',
    [\kappa'])$ exists provided that the edge $(d, u, k) \to (d', u', k)$ exists
    in $G_b$ for sufficiently large integers $k \equiv \kappa \pmod{L(b)}$ and
    $k' \equiv \kappa' \pmod{L(b)}$. 
\end{definition}

\begin{lemma}
    \label{degrees}
    $G_b'$ is well-defined and has both in- and out-degree at most $1$.
\end{lemma}

\begin{proof}
    The periodicity of \eqref{S-mod} for sufficiently large $k$ shows that the
    computation of $(d, u, k) \to (d', u', k')$ in $G_b$ depends only on $k
    \bmod L(b)$, so $G_b'$ is well-defined. It further shows that the in- and
    out-degree of $(d, u, [\kappa])$ in $G_b'$ will be the same as $(d, u, k)$ in
    $G_b$ for sufficiently large $k \equiv \kappa \pmod{L(b)}$. Vertices in
    $G_b$ have out-degree at most $1$ (by the lexicographically earliest
    construct) and in-degree at most $1$ (by Lemma~5.3 in \cite{neil}).
\end{proof}

\begin{proposition}
    Comma sequences with arbitrary initial conditions in base $b$ terminate if
    and only if $G_b'$ contains no cycles.
\end{proposition}

\begin{proof}
    The existence of infinite comma sequences is equivalent to the existence of
    infinite walks on the vertices of $G_b$, which is in turn equivalent to the
    existence of a cycle in $G_b'$.
\end{proof}

This proposition implies a relatively simple, direct way to prove that all base
$b$ comma sequences are finite: Check that $G_b'$ contains no cycles. The
runtime of this depends directly on $L(b)$, because Lemma~\ref{degrees} implies
that $G_b'$ is the disjoint union of independent paths and cycles. Our initial
calculations, limited to bases no higher than 23, used this idea. However, there
is a far better method.

\begin{proposition}
    Comma sequences with arbitrary initial conditions in base $b$ terminate if none of the points $(d, u, 0)$ lie on a cycle. 
\end{proposition}
Because edges in $G_b'$ are of the form
\begin{equation*}
    (d, u, \kappa) \to (d', u', \kappa + 1),
\end{equation*}
any cycle in $G_b'$ would have to contain every value of $\kappa \bmod L(b)$. In
particular, to check that no cycle exists, it suffices to check that no point
$(d, u, 0)$ leads to a cycle. We carried out this computation for $b \in \{3, 4,
\dots, 633\}$ and observed that there were no cycles in $G_b'$, thereby proving
that all comma sequences in these bases are finite.

\section{The geometric model}
\label{sec:simple}

\begin{definition}
    Let $D(b)$ be the number of comma sequences in base $b \geq 2$ with initial
    values in $[b^k - b^2, b^k)$ (with $k > 2$) that do not reach $b^k$.
    (As mentioned in the previous section, this quantity is independent of $k$.)
\end{definition}

The probabilistic model of \cite{neil} makes the simplifying assumption that
comma sequences enter the danger intervals $[b^k - b^2, b^k)$ uniformly
randomly, and that each interval should be considered independent. Under this
model we would expect to hit a landmine after $b^2 / D(b)$ intervals. In this
section we will prove that $D(b + 1) - D(b)$ is essentially \oeis{A136107}, the number
of ways to write $b + 1$ as the difference of positive triangular numbers. This
will establish the generating function identity \eqref{tri-gf}, and as a
corollary the asymptotic estimate $$\frac{b^2}{D(b)} \sim \frac{2b}{\log 2b}.$$
While this result is correct, the underlying model is slightly off. See
Section~\ref{graph-model} for better estimates.

For brevity, we will refer to the base $b$ integer $(b - 1, b - 1, \dots, b - 1,
x, y)_b$ as $\{x, y\}_b$.

\begin{lemma}
    \label{lemma-preds}
    For $y>0$, $\{x,y\}_{b - 1}$ is the parent of $\{r,s\}_{b - 1}$ in a comma
    sequence for base $b - 1$ if and only if $\{x + 1, y\}_b$ is the parent of
    $\{r+1, s\}_b$ in base $b$.
\end{lemma}

\begin{proof}
    By the comma rule, we have
    \begin{equation*}
        \{r, s\}_{b - 1} = \{x, y\}_{b - 1} + (y, b - 2)_{b - 1}.
    \end{equation*}
    There must be a carry in the units digit because $y$ is positive and $s$ is
    less than $b - 1$, so $y + (b - 2) = s + b - 1$. Because there is no carry
    in the tens digit this gives $x + y + 1 = r$. Together these imply
    $y+(b-1)=s+b$ and $(x+1)+y+1=(r+1)$, so $\{x+1,y\}_b$ is followed by $\{r+1,s\}_b$
    in base $b$. (Note that $r+1$ and $x+1$ are digits base $b$ since $r$ and
    $x$ are digits in base $b-1$.) The other direction follows the same logic. 
\end{proof}

\begin{lemma}
    \label{lemma-diff-escape}
    $D(b) - D(b - 1)$ is one more than the number of elements of the form $\{0,
    d\}_b$ with $0 < d < b - 1$ which fail to escape the base $b$ interval.
\end{lemma}

\begin{proof}
    If $\{r, s\}_{b - 1}$ is a landmine in base $b - 1$, then $\{r + 1, s\}_b$
    is a landmine in base $b$. By Lemma~\ref{lemma-preds}, if $x, y > 0$ and
    $\{x, y\}_{b - 1}$ leads to the landmine $\{r, s\}_{b - 1} \neq \{x, y\}_{b
    - 1}$ in base $b - 1$, then $\{x + 1, y\}_b$ leads to $\{r + 1, s\}_b$ in
    base $b$. Together these account for $D(b - 1)$ failed escapes. Elements of
    the form $\{x,0\}_b$ are followed by $\{x,b-1\}_b$ (with a difference of
    $(0, b-1)_b$) and those escape with a difference of $(b - 1, 1)_b$.
    Therefore, the only possible ``new'' escapes are (1) the landmine $\{1, b -
    2\}_b$ not covered by the map $\{r, s\}_{b - 1} \mapsto \{r + 1, s\}_b$ and
    (2) elements of the form $\{0, y\}_b$ with $0 < y < b - 1$.
\end{proof}

\begin{lemma}
    \label{lemma-triangular}
    With $0 < y < b - 1$, the element $\{0, y\}_b$ fails to escape the interval
    in base $b$ if and only if there is a positive integer $n$ such that
    \begin{equation}
        \label{b-eqn}
        b = \frac{(n + 1)(2(y + 1) - n)}{2},
    \end{equation}
    with a single exception if $b$ is a triangular number. If $b = (k + 1)(k +
    2)/2$ for some integer $k$, then the element $\{0, k\}_b$ escapes despite the
    solution $n = y = k$.
\end{lemma}

\begin{proof}
    If we continue the comma sequence beginning at $\{0, y\}_b$, then as long as we
    are in the interval we obtain the sequence of values $\{x(n), y(n))\}_b$ which
    satisfy $\{x(0), y(0)\}_b = \{0, y\}_b$ and the recurrences
    \begin{align*}
        y(n + 1) &= y(n) - 1 \\
        x(n + 1) &= x(n) + y(n) + 1.
    \end{align*}
    From these recurrences it is easy to prove
    \begin{equation*}
        x(n) + y(n) = (n + 1) y - \frac{n(n - 1)}{2},
    \end{equation*}
    and so $x(n) + y(n) = b - 1$ is equivalent to \eqref{b-eqn}.

    If $\{0, y\}_b$ hits a mine in $n$ steps then $x(n) + y(n) = b - 1$, which
    proves half of our lemma. For the converse, suppose that there is a positive
    integer $n$ for which \eqref{b-eqn} has a solution. Then we have $x(n) +
    y(n) = b - 1$, but this does not mean that the sequence terminates. The
    possibilities are as follows:
    \begin{enumerate}
        \item We hit a mine.
        \item We hit $\{0, b - 1\}_b$. We excluded this with $d < b - 1$.
        \item We hit $\{b - 1, 0\}_b$. In this case, we have $y(n) = n - y = 0$, so
            $n = y$. Plugging this into \eqref{b-eqn} shows that
            \begin{equation*}
                b = \frac{(n + 1)(n + 2)}{2}.
            \end{equation*}
            Therefore when $b$ is a triangular number, there is precisely one
            $y$ which leads to a solution to \eqref{b-eqn}, but should not be
            counted.
    \end{enumerate}
    It follows that solutions to \eqref{b-eqn} determine all elements $\{0,
    y\}_b$ which fail to escape the interval, except for when $b$ is triangular.
\end{proof}

\begin{lemma}
    \label{lemma-quadratic}
    The number of $0 < y < b - 1$ for which there are positive integer
    solutions to \eqref{b-eqn} is one less than the number of odd divisors of $b$.
\end{lemma}

\begin{proof}
    Note that \eqref{b-eqn} is equivalent to the following quadratic in $n$:
    \begin{equation*}
        n^2 - (2y + 1)n + 2(b - 1 - y).
    \end{equation*}
    Because the linear term is odd and the constant term is positive, for any
    $y$ for which there are integer solutions, there will in fact be
    \emph{two}, one even and one odd. Thus the number of $y$ which have
    solutions is in correspondence with the number of even $n$ which lead to
    integer solutions. If $n$ is even then \eqref{b-eqn} holds if and only if
    $n + 1$ divides $b$, which establishes a correspondence with the odd
    divisors of $b$. Note that $b-1$ will always be a solution and $0$ will never be. 
\end{proof}

\begin{proposition}
    $D(b) - D(b - 1) = A136107(b)$ for $b \geq 3$.
\end{proposition}

\begin{proof}
    By Lemma~\ref{lemma-diff-escape}, $D(b) - D(b - 1)$ is one more than the
    number of elements $(0, d)$ which fail to escape the interval in base $b$.
    Lemmas~\ref{lemma-triangular} and \ref{lemma-quadratic} tells us that this
    is the number of odd divisors of $b$, except when $b = k(k - 1) / 2$ is a
    triangular number, in which case we must discard a single divisor. According
    to a comment by  R.~J.~Mathar this is equivalent to $A136107(b)$.
\end{proof}

\begin{proof}[Proof of Theorem~\ref{theorem-difference}]
    The previous proposition implies
    \begin{equation*}
        D(b) = \sum_{k = 1}^b A136107(b) - 1
    \end{equation*}
    for $b \geq 2$. According to a comment by Vladeta Jovovic in the OEIS, the
    ordinary generating function for A136107 is
    \begin{equation*}
        \sum_{n \geq 1} \frac{t^{n(n + 3)/2}}{1 - t^n}.
    \end{equation*}
    Using the formula for the generating function of partial sums \cite{wilf},
    we see that the ordinary generating function of $D(b)$ is
    \begin{equation*}
        \frac{1}{1 - t} \sum_{n \geq 1} \frac{t^{n(n + 3)/2}}{(1 - t^n)} - \frac{t^2}{1 - t},
    \end{equation*}
    which matches \eqref{tri-gf}.
\end{proof}

\section{A more accurate model}
\label{graph-model}

In Section~\ref{details} we constructed a finite graph $G_b'$ which contains no
cycles if and only if all comma sequences in base $b$ are finite. By randomly
sampling points from this graph, we obtain a different probabilistic model of
comma sequence behavior. This model predicts the lengths of comma sequences more
accurately than the geometric model in \cite{neil}, provided that there truly
are no cycles.

The expected number of danger intervals which a comma sequence survives can
reasonably be computed as
\begin{equation*}
    |V(G_b')|^{-1} \sum_{v \in V(G_b')} h(v),
\end{equation*}
where $h(v)$ is the number of edges between $v \in V(G_b')$ and the end of the
path in which it resides, and $V(G_b')$ is the set of all vertices in $G_b'$. A
simple calculation shows that this equals
\begin{equation*}
    |V(G_b')|^{-1} \sum_P {|P| \choose 2},
\end{equation*}
where the sum is over all paths $P$ in $G_b'$ and $|P|$ is number of vertices in
a path. By adjusting the denominator, we can rephrase this as an expectation in
terms of a uniformly randomly chosen path $P$:
\begin{align*}
    |V(G_b')|^{-1} \sum_P {|P| \choose 2}
    &=
    \frac{\text{\# of paths}}{|V(G_b')|} \mathbb{E} \left[ {|P| \choose 2} \right].
\end{align*}
Recall that all paths in $G_b'$ begin at a point of the form $(d, 0, \kappa)$
with $\kappa$ taken mod $L(b)$ (as in \eqref{lcm}), so there are $(b - 2) L(b)$
paths in total.

Empirical testing led us to the following conjecture.

\begin{conjecture}
    \label{path-conjecture}
    As $b \to \infty$, the path length $|P|$ is approximately exponentially
    distributed with rate close to $(b / 2 + 1)^{-1}$. In particular, $\mathbb{E}[|P|]
    \approx b / 2 + 1$ and $Var(|P|) \approx (b / 2 + 1)^2$.
\end{conjecture}

With this conjecture and the preceding remarks, we may very roughly estimate the
number of danger intervals survived as
\begin{equation}
    \frac{b}{2} + 1 + O(b^{-1}).
\end{equation}
The geometric model of \cite{neil} predicts that roughly $2b / \log(2b)$
intervals would be survived. Figure~\ref{interval-estimates} gives evidence that
$b / 2 + 1$ is a better estimate. It also suggests that $b / 2 + 1$ may also be
slightly smaller than the true mean.

\begin{figure}[t]
    \centering
    \includegraphics[width=\textwidth]{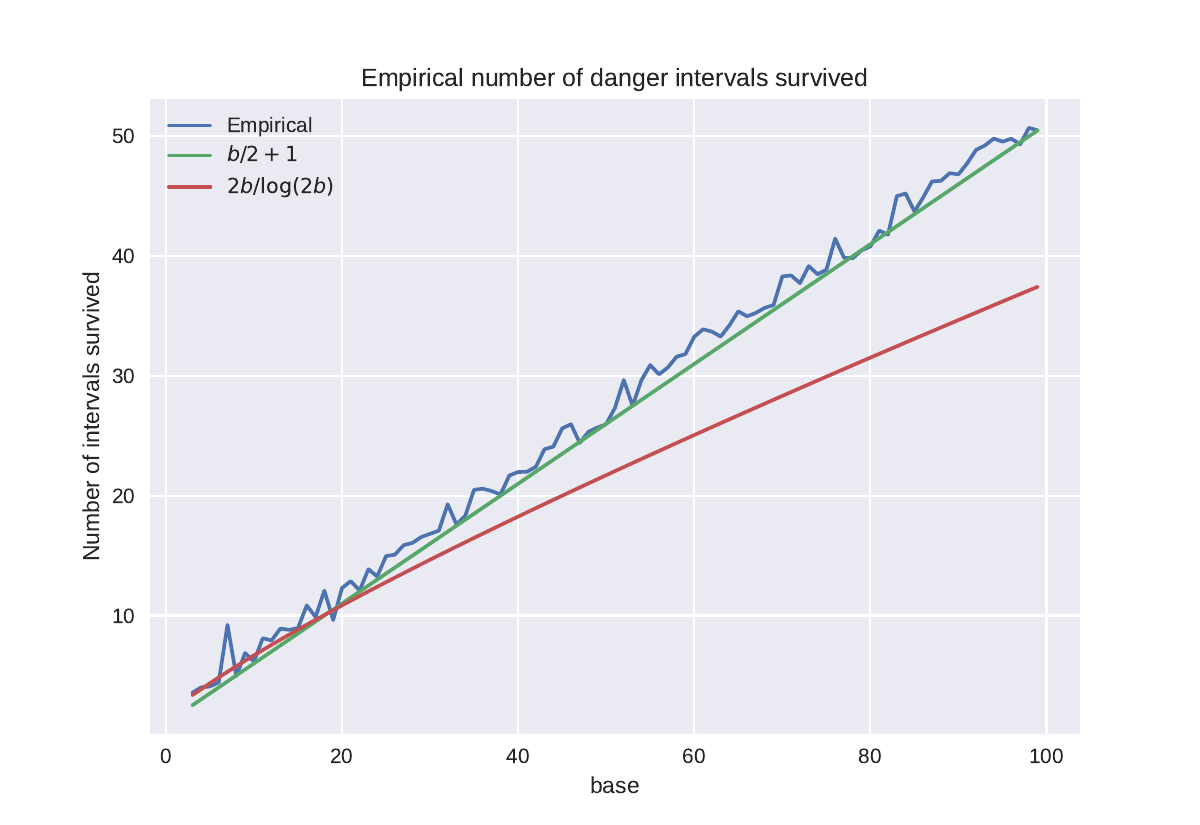}
    \caption{Average number of danger intervals survived for comma sequences in
    base $b$, estimated empirically using the initial values $1$, $2$, \dots,
    $b^2 - 1$. Note that the empirical curve begins close to $2b / \log(2b)$,
    but pulls away for later bases.}
    \label{interval-estimates}
\end{figure}

\section{Computational details}
\label{compdetails}

The considerations of Section~\ref{details} lead to Algorithm~\ref{isFinite}.
Our parallel implementation is in the Go programming language \cite{go, github}.

The runtime of our algorithm depends, roughly speaking, on the length of the
path that each point $(d, u, 0)$ lies on. In the previous section we conjectured
that paths have an average length of $b / 2 + 1$, so we expect the average
runtime to be around $O(b^4)$, without any further optimizations. We can obtain
constant time improvements by noticing that the points $(d, u, 0)$ are processed
independently, so the procedure \texttt{IsFinite} can be trivially parallelized.

\vspace{0.1in}

\begin{algorithm}
\caption{Check whether base $b$ comma sequences are finite.}
\label{isFinite}
\hrule
\begin{algorithmic}
    \Procedure {IsFinite}{$b$}
        \For {$1 \leq d < b$, $u \in U(b, d)$}
            \State $P, P_0 \gets (d, u, 0)$
            \Repeat
            \State $P \gets$ \Call{AdvancePoint}{$b$, $P$}
            \If {$P = P_0$}
                \State \Return{infinite}
            \EndIf
            \Until{isMine($P$)}
        \EndFor
        \State \Return{finite}
    \EndProcedure
    \State \hrule
    \State
    \Procedure{AdvancePoint}{$b, d, u, k$}
        \If {$d = b - 1$}
            \State $d' \gets 1$
            \State $k' \gets (k + 1) \bmod L(b)$
        \Else
            \State $d' \gets d + 1$
            \State $k' \gets k$
        \EndIf
        \State $S \gets \Call{Cycle}{d, u}$ \Comment{see \eqref{cycle}}
        \State $(k_0, l) \gets \Call{Order}{b, \sum S}$ \Comment{$b^k$ is periodic with period $l$ for $k \geq k_0$}
        \State $gap \gets (b^{k_0 + (k - k_0) \bmod l} + u) \bmod \sum S$
        \If {$gap = 0$}
            \State $gap \gets \Call{LastElem}{S}$
            \Comment {the shortcut does not apply if the gap is 0}
        \Else
            \State $i \gets 0$
            \While {$gap > S[i]$}
                \If {$gap < b^2$ and $d = b - 1$ and \Call{IsMine}{$b$, $(d', gap, k')$}}
                    \State \textbf{break}
                \EndIf
                \State $gap \gets gap - S[i]$
                \State $i \gets i + 1$
            \EndWhile
        \EndIf
        \State $u' \gets gap$
        \State \Return $(d', u', k')$
    \EndProcedure
\end{algorithmic}
\end{algorithm}

More importantly, we can save a lot of time by noticing that
\texttt{AdvancePoint} duplicates effort over repeated calls. The computations in
that procedure depends only on $d$, $u$, and $k \bmod l(d, u)$, the periods
appearing in Equation~\eqref{lcm}. So, while \texttt{IsFinite} must consider $k
\bmod L(b)$, we can cache the results of \texttt{AdvancePoint(d, u, k mod l(d,
u))}. This cache requires a memory footprint of roughly
\begin{equation*}
    O \left((\log b)^3 \sum_{d, u} l(d, u) \log l(d, u) \right),
\end{equation*}
which has not been a bottleneck thus far. (Base $b = 629$ required approximately
1.9GB of memory.)

As a minor remark, it proved difficult to find a good reference algorithm for
determining precisely when and with what period $b^k \bmod m$ repeats for
\emph{arbitrary} integers $b$ and $m \geq 1$. For completeness, this is how it
is done:
\begin{itemize}
    \item If $\gcd(b, m) = 1$, then the period equals the smallest positive $k$
        such that $b^k \equiv 1 \pmod{m}$, and the cycle begins immediately
        \cite[ch.~6]{hardywright}.

    \item If $\gcd(b, m) = g > 1$, then let $k_0$ be the smallest positive
        integer such that $\gcd(b^k, m)$ is constant for $k \geq k_0$. Then
        $b^k \equiv b^{k + l} \pmod{m}$ for $k \geq k_0$ if and only if $b^l
        \equiv 1 \pmod{m / g}$. Now apply the previous case to $b$ and $m / g$.
\end{itemize}

\paragraph{Acknowledgements} We thank Neil Sloane and Giovani Resta for bringing
this problem and several key ideas to our attention, and Luke Pebody for his
suggestion of a far superior way to search for cycles in $G_b'$. We also thank
the Office of Advanced Research Computing (OARC) at Rutgers, The State
University of New Jersey for providing access to the Amarel cluster and
associated research computing resources (\url{https://oarc.rutgers.edu}), as
well as the office of Research Computing at Dartmouth College for access to the
Discovery Cluster (\url{https://rc.dartmouth.edu/}).

\printbibliography

\end{document}